\documentclass{article}

\usepackage{amsmath,latexsym,amssymb}
\usepackage[german,french,english]{babel}
\usepackage[all]{xy}
\usepackage{mathrsfs}
\usepackage{graphicx}
\usepackage{color}  

\usepackage[applemac]{inputenc}

\newcommand{\maxi}{\mathfrak{M}}
\newcommand{\mini}{\mathscr{M}}

\newtheorem{thm}{Theorem}
\newtheorem{example}{Example}
\newtheorem{prop}{Proposition}[section]
\newtheorem{lemma}[prop]{Lemma}
\newtheorem{cor}[prop]{Corollary}
\newtheorem{rmk}[prop]{Remark}

\newtheorem{main}{Theorem}

\newtheorem{defi}[prop]{Definition}

\newcommand{\proof}[1][]{{\it Proof#1: }}
\newcommand{\qed}[1][3mm]{\hspace*{\fill} $\Box$ \vspace{#1}}

\newcommand{\Tg}{\mathcal{T}_g}

\newcommand{\ZZ}{{\mathbf Z}}
\newcommand{\CC}{{\mathbf C}}
\newcommand{\PP}{{\mathbf P}}

\newcommand{\FF}{{\mathbb F}}
\newcommand{\homeo}{\mathcal{H}omeo}
\newcommand{\bff}{\mathbb{F}}

\newcommand{\s}{\sigma}

\newcommand{\tto}{\longrightarrow}

\newcommand{\surj}{\to\!\!\!\!\!\!\!\!\tto}

\newcommand{\inv}{^{^{-1}}}

\newcommand{\dfami}{{\cal D}}

\newcommand{\hfami}{{\cal H}}

\newcommand{\ofami}{{\cal O}}

\newcommand{\tfami}{{\cal T}}

\newcommand{\map}{\mathcal{M}ap}

\newcommand{\cutoff}[1]{}

\newcommand{\Xmg}{X_{M,g}}
\newcommand{\Xmgp}{X_{M,g}'}
\newcommand{\omg}{\Omega^{M}_{2,g+2}}

\newcommand{\calx}{\mathcal{X}}
\newcommand{\cald}{\mathcal{D}}
\newcommand{\ko}{\mathcal{O}}
\newcommand{\ra}{\rightarrow}

\begin{document}

\title{Mapping Class Groups of Trigonal Loci.}



\author{Michele Bolognesi and Michael L\"onne}

\maketitle

\begin{abstract}
In this paper we study the topology of the stack $\Tg$ of smooth trigonal curves of genus $g$, over the complex field. We make use of a construction by the first named author and Vistoli, that describes $\Tg$ as a quotient stack of the complement of the discriminant. This allows us to use techniques developed by the second named author to give presentations of the orbifold fundamental group of $\Tg$, of its substrata with prescribed Maroni invariant and describe their relation with the mapping class group $\map_g$ of Riemann surfaces of genus $g$.
\end{abstract}


\section{Introduction}

The theory of moduli spaces is one of the most charming subjects of algebraic geometry. Already at the very first stages of its development, it seemed clear that a good way to construct \it spaces \rm that would solve in some sense a moduli problem was to display them as quotients by group actions. 
As everybody learns in a first course of GIT, taking quotients is a delicate operation in algebraic geometry, but a good solution, at least in moduli theory, has been given by the theory of quotient stacks. If an algebraic stack $S$ is a quotient stack $[X/G]$, where $G$ is an algebraic group acting on an algebraic variety $X$, then its geometry is very much related to the action of $G$ on $X$.
The first example that comes to mind is the stack $\mathcal{M}_{1,1}$ of elliptic curves, which is a quotient $[X/\mathbb{G}_m]$, where $X$ is the complement in $\mathbb{A}^2$ of the discriminant
hypersurface $4x^3 + 27y^2 = 0$, and $\mathbb{G}_m$ acts with weights 4 and 6.
Very powerful techniques have been developed in equivariant intersection theory, after the landmark work of Edidin and Graham \cite{edidin-graham}, and applications have flourished (e.g. \cite{vistoli-m2,arsie-vistoli,viviani-fulghesu}) in equivariant intersection theory.

\medskip

The goal of this paper is to explore a particular quotient stack, the stack of smooth trigonal curves, under the somewhat different light of homotopy groups.
This stack has been constructed in \cite{bolognesi-vistoli} and it has a presentation as a quotient stack $[X'/\Gamma_g]$, where $\Gamma_g$ is a certain algebraic group, and $X'$ is an open set inside the total space of a vector bundle over an open subset of a representation of $\Gamma_g$ (see Sect. \ref{sect-trigostack} for more details). In particular we will concentrate on the study of the orbifold fundamental group (see Sect. \ref{sect-orbifund} for a detailed definition) of the stack $\Tg$.

\medskip

The universal family of curves over $\Tg$ has a structure of fibre bundle
for the group $\homeo_g$ of orientation preserving homeomorphisms of a Riemann surface of genus $g$. This allows us to define a \it monodromy map \rm 

$$\mu_g^{\mathcal{T}}: \pi_1^{orb}(\Tg,x_0) \to \pi_0(\homeo_g)=\map_g,$$

given a base point $x_0\in \Tg$.
The analogous map for the stack of hyperelliptic curves maps
to the proper subgroup of hyperelliptic mapping classes.
But in Theorem \ref{surjmontrigo} we show the following quite surprising statement.

\begin{main}
The monodromy map $\mu_g^{\mathcal{T}}$ is surjective.
\end{main}

Now recall that $\Tg$ admits a classical stratification in terms of Maroni invariant \cite{maroni}. Let $\Tg^M$ denote the stratum with Maroni invariant equal to $M$. Then one can ask what kind of map does the natural inclusion $\Tg^M\hookrightarrow \Tg$ induce on the orbifold fundamental groups. The answer to this question is Theorem \ref{isomaroni}.

\begin{main}
If $M<g/3-1$, the inclusion 
$\Tg^M \hookrightarrow \Tg$ induces a surjection

$$\pi_1^{orb}(\Tg^M) \surj \pi_1^{orb}(\Tg)$$

with kernel a homomorphic image of $\ZZ/(M)$.
\end{main}

As a corollary of this theorem we obtain that the restriction of the monodromy map $\mu_g^{\Tg}$ to $\Tg^M$ is still surjective, except for maximal Maroni invariant and $g\equiv 1 (mod\ 3)$.

\medskip

Finally we derive an explicit finite presentation for the orbifold
fundamental group of the maximal Maroni stratum $\Tg^{\maxi}$.
Recall that Dolgachev and Libgober \cite{dolga-ligbo}
posed the vastly open problem to determine the fundamental
group of the discriminant complement of any (complete) linear
system.

They handle the case of linear systems of elliptic curves on $\PP^2$
and $\PP^1\times\PP^1$ as well as linear systems on curves, 
but actually the first result of that kind is due to Zariski
who considered the complete linear systems on $\PP^1$, which
he showed to have fundamental group given by the braid group
with one additional relation.

Libgober \cite{lib-cubic} and Looijenga \cite{loo-artin}
later considered orbifold fundamental groups which they showed
to be natural quotients of groups finitely presented of Artin type.

Building on the results of \cite{loenne-duke} and \cite{loenne-weierstrass}, we manage to obtain a presentation which
again follows that pattern:

\begin{main}
The orbifold fundamental group 
$\pi_1^{orb}(\Tg^{\maxi})$ of the maximal Maroni stratum, in the case $g \equiv 1\ (\mod\ 3)$ has a presentation in terms of generators $t_1,\dots, t_{2g+2}$, and relations:

\begin{itemize}
\item of "diagram type"
\begin{eqnarray*}
t_it_jt_i = t_jt_it_j & \text{if }\ j=i+1\ \text{ or }\ j=i+2;\\
t_it_j=t_jt_i & otherwise, 
\end{eqnarray*}

and

\begin{equation*}
(t_it_jt_i^{-1})t_k=t_k(t_it_jt_i^{-1})\ \text{ if }\ i+1=j=k-1. 
\end{equation*}




\item of "global type":
denote $\delta_0=t_1t_2\: t_3t_4 \cdots t_{2g+1}t_{2g+2}$

\begin{eqnarray*}
\delta_0 & \text{centralizes} & t_{2g+1} t_{2g-1} \cdots t_3 t_1, \\
&& t_{2g+2} t_{2g} \cdots t_4 t_2.
\end{eqnarray*}

\item of "quotient type":
denote $\delta_1=t_{2g+1}t_{2g+2}\: 
t_{2g-1} t_{2g} \cdots t_1t_2.$

\begin{eqnarray*}
(\delta_0\delta_1)^3 & = & 1\\
\delta_0^{g+2} &= &1. 
\end{eqnarray*}

\end{itemize}
\end{main}

See Section \ref{sect-presentation} for more details on the different relations.

\bigskip

\bf Acknowledgments: \rm We warmly thank C.Ciliberto, F.Flamini, T.Dedieu, I.Tyomkin and A.Vistoli for suggestions and fruitful email exchange.

\subsection{The orbifold fundamental group}\label{sect-orbifund}

We will denote by $[X/G]$ the (possibly orbifold) quotient of a topological space (or variety, scheme, etc.) $X$ by a group $G$. When we work on an orbifold $[X/\Gamma]$, the orbit of $x_0\in X$ will be indicated by $\Gamma \cdot x_0$.

\begin{defi} \rm\label{orbifundamentalgp}
Let $X$ be a topological space, $x_0$ a point of $X$, $\Gamma$ a group acting on $X$. Let moreover $E\Gamma$ be the universal $\Gamma$-principal bundle over a classifying space $\mathbf{B}\Gamma$. We define the orbifold fundamental group $\pi_1^{orb}([X/\Gamma],\Gamma\cdot x_0)$ as the classical fundamental group $\pi_1([X\times E\Gamma/\Gamma],\Gamma \cdot (x_0,t))$, where $t$ may be any element of $E\Gamma$, since $E\Gamma$ is contractible.
\end{defi}

\begin{rmk}
The orbifold fundamental group then fits into a commutative diagram
\[\xymatrix{\pi_1(X,x_0)\ar[r] \ar@{=}[d] & \pi_1^{orb}([X/\Gamma], \Gamma\cdot x_0) \ar[r] \ar@{=}[d] & \pi_0(\Gamma,1) \\
\pi_1(X\times E\Gamma ,(x_0,e))\ar[r] & \pi_1([X\times E\Gamma/\Gamma], \Gamma\cdot (x_0,t)) \ar[r] &  \pi_0(\Gamma,1). \\ }\]
\end{rmk}

Accordingly, given a morphism $\phi$ from a smooth algebraic variety $Y$, pointed at $y_0$, to an orbifold $[X/\Gamma]$, we define a map of fundamental groups as follows. The morphism $\phi$ is given by the datum consisting of a $\Gamma$-torsor $P_Y$ over $Y$ and a $\Gamma$-equivariant morphism  $\tilde\phi: P_Y \to X$.
By choosing a base point $\tilde y\in P_Y$ over $y_0$ and its images
$\tilde\phi (\tilde y)\in X$ and $\Gamma \cdot \tilde\phi (\tilde y)\in [X/\Gamma]$, we get a commutative diagram


\begin{equation}
\xymatrix{ \pi_1(P_Y, \tilde y) \ar[r]^{\tilde\phi_* } \ar@{->>}[d] & \pi_1(X,\tilde\phi (\tilde y))\ar@{->>}[d]   \\
\pi_1(Y,y_0) \ar[r]^-{\phi_*} &  \pi_1^{orb} ([X/\Gamma],\Gamma \cdot \tilde\phi (\tilde y)) \\ }
\end{equation}

Now, let us recall the construction from Def.\ \ref{orbifundamentalgp}. Since $E\Gamma$ is contractible, then $\pi_1^{orb}([X/\Gamma], \Gamma \cdot x_0)= \pi_1(X,x_0)/\pi_1(\Gamma,e)$. This equality is well-defined and independent of the choice of the base point inside $\Gamma\cdot x_0$, in the sense that there exists a canonical isomorphism $\psi_{x_0,x_1}:\pi_1(X,x_0)\to\pi_1(X,x_1)$,  $\forall x_0,x_1 \in \Gamma\cdot x_0$. Hence, the extension at the bottom of the above diagram - which defines $\phi_*$ - is possible in a unique way and independently of the choice of $\tilde y$, since the kernel of both vertical maps is the image of $\pi_1(\Gamma,e)$.

\subsection{The stack of smooth trigonal curves}\label{sect-trigostack}

\newcommand{\omegafull}{\Omega^{full}_{2,g+2}}
\newcommand{\sym}{\mathrm{Sym}}

The purpose of this subsection is to review and recall the construction of the quotient stack of smooth trigonal curves of genus $g$, as introduced in \cite{bolognesi-vistoli} generalizing a construction by Miranda \cite{miranda}. This stack is constructed as the complement of an invariant hypersurface inside a quotient stack of a vector bundle over an open set of an affine space. On the other hand, it also has a presentation as an orbifold itself. Let us review these constructions.

\newcommand{\wt}{\widetilde}
\newcommand{\Mat}{M\!at}

\medskip
 
We warn the reader that our notation is slightly different from that of \cite{bolognesi-vistoli}. Let us recall from \cite{miranda} that the datum of a trigonal curve $t:C\to \PP^1$ of genus $g$ is equivalent to the datum of a rank two vector bundle $E$ on $\PP^1$ (actually obtained as $t_*\ko_C/ \ko_{\PP^1}$, and known as \it Tschirnhausen module \rm) with a few properties, and a section of $\sym^3E\otimes \det E^*$. Notably, the splitting type $(m,n)$ of $E$ should be such that $m + n = g + 2$, and, if $C$ is integral, then $m,n \geq \frac{g+2}{3}$ (see also \cite[Prop. 2.2]{bolognesi-vistoli}). The stack of smooth trigonal curves is constructed starting from this datum.

Let $\Mat_{2,g+2}$ be the affine space of $(g+2)\times (g+4)$ matrices $(l_{ij})$, where each $l_{ij}$ is a linear form in two indeterminates. Let us denote by $\wt{\Omega}_{2,g+2}$ the open subscheme of $\Mat_{2,g+2}$ parametrizing matrices $(l_{ij})$ with the property that the matrix $(l_{ij}(p))$ has rank $2$ at all points $p\in\PP^1$. As remarked in \cite[Prop. 4.2]{bolognesi-vistoli}, 

\begin{lemma}\label{purecod}
The complement $\Mat_{2,g+2}/\wt{\Omega}_{2,g+2}$ is pure-dimensional of codimension at least 2.
\end{lemma}

In what follows, we will identify a matrix $(l_{ij})$ with the associated sheaf homomorphism over $\PP^1 \times \wt{\Omega}_{2,g+2}$

\begin{equation}\label{universalseq}
\xymatrix{ \ko_{\PP^1\times \wt{\Omega}_{2,g+2}}(-1)^{g+2}\quad\ar[r]^{\quad\: (l_{ij})} & \quad \ko^{g+4}_{\PP^1 \times \wt{\Omega}_{2,g+2}}.\\ }
\end{equation}

We will denote by $E_{2,g+2}$ the cokernel of the above universal morphism. It is locally free of rank 2. Moreover, inside $\wt{\Omega}_{2,g+2}$ we will distinguish an open subset $\Omega_{2,g+2}$, which is defined as follows:
for any matrix $(l_{ij}) \in \Omega_{2,g+2}$, the cokernel sheaf $E_{2,g+2}$ is a globally generated locally free sheaf of rank 2, the degree of its restriction to the geometric fibers of the projection $\pi:\PP^1\times \Omega_{2,g+2} \to \Omega_{2,g+2}$ is $(g+2)$, and if $(m,n)$ is the splitting type of $E_{2,g+2}$ over such a fiber, then $m, n\geq \frac{g+2}{3}$. We will often abuse of notation by denoting simply $E_{2,g+2}$ the restriction to $\Omega_{2,g+2}.$ 
Let us also introduce the well-known Maroni invariant of a trigonal curve $C$ as $M:=|m-n|$. This is a discrete invariant of a trigonal curve (introduced in \cite{maroni}) that takes the even values from 0 to $\lfloor\frac{g+2}{3}\rfloor$ if $g(C)$ is even, or the odd values from 1 to $\lfloor\frac{g+2}{3}\rfloor$ if $g(C)$ is odd. The open subset $\Omega_{2,g+2}$ is naturally stratified by the Maroni invariant $\Omega_{2,g+2}=\Omega^{\mini}_{2,g+2}\cup \cdots \cup \Omega^{M}_{2,g+2}\cup \cdots \cup \Omega^{\maxi}_{2,g+2}$, where $\mini$ and $\maxi$ denote respectively the minimal and the maximal Maroni invariant. The stratification via the Maroni invariant extends to the whole moduli space of trigonal curves.

\medskip

In order to pass from $\Omega_{2,g+2}$ to trigonal curves we need to consider a tensor of $E_{2,g+2}$, namely $\sym^3E_{2,g+2} \otimes \det E_{2,g+2}^*$, and consider the sheaf $\mathcal{E}_{2,g+2}:=\pi_*(\sym^3E_{2,g+2} \otimes \det E_{2,g+2}^*)$ over $\Omega_{2,g+2}$. This is locally free and its formation commutes with base change. We will call $X_g$ the total space of the vector bundle corresponding to $\mathcal{E}_{2,g+2}$. We now introduce (see \cite[Sect. 4]{bolognesi-vistoli}) an algebraic group action on $X_g$. Let us take $G:=GL_{g+4} \times GL_{g+2}\times GL_2$ and consider the embedding

\begin{eqnarray}
\eta:\quad \mathbb{C}^* &\hookrightarrow & GL_{g+4} \times GL_{g+2}\times GL_2,\\
t\:\: & \mapsto & (Id_{g+4},\:\:\:tId_{g+2},\:\:t^{-1}Id_2). \notag
\end{eqnarray}

of the torus. The image of this embedding is a central group subscheme of $G$ and we will denote by $\Gamma_g$ the cokernel of $\eta$. The group $\Gamma_g$ acts naturally on $\Omega_{2,g+2}$ and $X_g$. In \cite[Thm. 5.3]{bolognesi-vistoli} the following is proven.

\begin{thm}\label{trigonalquot}
Let $\hat{\Tg}$ be the moduli stack consisting of objects $(C \stackrel{t}{\ra}P \ra S)$, where $P$ is a $\PP^1$-bundle over $S$, $t$ is a triple cover such that the splitting type $(m,n)$ of the associated rank 2 bundle on $P$ satisfies $m,n \geq \frac{g+2}{3}$. Then there is an equivalence of fibered  categories $\hat{\Tg}\cong [X_g/\Gamma_g]$
\end{thm}

\newcommand{\Dg}{\mathcal{D}_g}
\newcommand{\Dmg}{\mathcal{D}_{M,g}}

We recall from \cite{bolognesi-vistoli} that a point of $\Omega_{2,g+2}$ completely defines a Tschirnhausen module $E$ over $\PP^1$ (up to the action of $\Gamma_g$), whereas it is the fibers of the vector bundle $X_g$ that naturally parametrize all the sections of $\sym^3E\otimes \det E^*$.

In order to consider the moduli stack $\Tg$ of \textit{smooth} trigonal curves we need to consider the complement of a certain closed integral substack $\mathcal{S}_g\subset \hat{\Tg}$ parametrizing singular curves. The substack $\mathcal{S}_g$ is of the form $[\mathcal{D}_g/\Gamma_g]$ for an invariant discriminant hypersurface $\mathcal{D}_g\subset X_g$. We will denote by $X_g'$ the complement $X_g - \cald_g$. It is not hard to observe that the restriction of $\Dg$ to the geometric fibers of the vector bundle $X_g$ corresponds naturally to the discriminant locus of the given space of sections over $\PP^1$, or curves inside the Hirzebruch surface $\mathbb{F}_M$.

\begin{thm}[\cite{bolognesi-vistoli}]\label{finetrigo}
The (fine) moduli substack $\Tg \subset \hat{\Tg}$ of smooth trigonal curves is equivalent to the quotient stack $[X'_g/\Gamma_g]$.
\end{thm}

\textbf{Notation}
We will denote by $\Omega^M_{2,g+2}$ the stratum  inside $\Omega_{2,g+2}$ corresponding to the Maroni invariant $M$. By restriction of $X_g$ we define a fibration $\Xmg \to \Omega^M_{2,g+2}$. We will denote by $\Dmg$ and $\Xmgp$ the corresponding discriminant and its complement.


\section{The monodromy map}

In this section we investigate the monodromy map associated to
families of trigonal curves. Though often defined using an
Ehresmann connection on a differentiable fibre bundle, it can be put
on a purely topological footing. We need to employ a more detailed
definition later, but conceptually monodromy does the following:
to a closed path in the base of a fibre bundle it associates the
homeomorphism of the base fibre obtained by a bundle trivialization
along the path. Its isotopy class only depends on the homotopy
class of the path, hence monodromy provides a well-defined map
from the fundamental group of the base to the mapping class
group of the fibre.

In our situation of trigonal families the monodromy map takes
values in the mapping class group of a curve of genus $g$.
Our aim now is threefold: first to define a monodromy map
on the fundamental group of the moduli orbifold, second
to show its universality, i.e.\ that it factors every monodromy map
of a trigonal family in a way to be made precise below, and last
to give properties of the image\footnote{The study of the kernel 
will be taken up in a subsequent paper}.

In our argument we need a notion of monodromy in the more
general setting of $G$-principal bundles - the usual situation corresponding
to $G$ equal to the group of homeomorphisms of the fibre. Let $I$ be the unit interval.

\begin{defi} \rm
Suppose $E\to B$ is a $G$-bundle over a base $B$ pointed at $b_0$,
then the \emph{monodromy map} (see Appendix A for the precise definition)

\[
\mu_G: \quad \pi_1(B,b_0)\quad \tto \quad \pi_0(G)
\]

associates to a closed path the isotopy class of an element $g$ of $G$. A trivialization along the path gives a map $g_I: I\to G$ and $g= g_I(1)$.
\end{defi}

Note that this general notion is well-adapted to comparing monodromy maps of bundles with possibly different structure groups. This feature will be relevant in the rest of the paper.

\begin{example}
There exists a tautological genus $g$ trigonal family over $X'_g$, obtained by pull-back along the quotient map of the universal family on $[X'_g/\Gamma_g]$. It is a fibre bundle for the group $\mathcal{H}omeo_g$ of orientation preserving homeomorphisms of a Riemann surface of genus $g$, hence the monodromy
associated to a base point $x_0\in X_g'$ is

\[
\mu^{X'}_g: \quad
\pi_1(X_g', x_0) 
\quad\tto\quad
\pi_0(\mathcal{H}omeo_g)=\mathcal{M}ap_g
\]

with values in the mapping class group $\map_g$ of genus $g$.
\end{example}

We make the following observation:

\begin{prop}
\label{trivially}
If $x_0\in X_g'$ is any base point and

\[
\pi_1(\Gamma_g,e) \quad \tto
\quad \pi_1(X'_g,x_0)
\]

the map induced by the $\Gamma_g$-action on $X'_g$,
then $\mu^{X'}_g $ is trivial on the image.
\end{prop}

\proof
It suffices to see that the restriction to the $\Gamma_g$-orbit 
$\Gamma_g \cdot x_0$ has trivial monodromy. But this is immediate
from the fact that the orbit consists of different ways to describe one
trigonal curve in a ruled surface by putting coordinate systems
on the base and on the surface.
\qed

Instead of giving a general definition of monodromy associated to
a moduli orbifold, we rather provide an ad hoc definition in our
special situation. Since in our situation the orbifold fundamental group is the cokernel of the inclusion above, we may define:

\begin{defi} \rm
The \emph{monodromy map} of the moduli orbifold is

\begin{equation}
\mu_g^{\mathcal{T}}: \quad \pi_1^{orb} ([X_g'/\Gamma_g],\Gamma_g\cdot x_0)
\:=\: \pi_1(X_g',x_0)\big/ \pi_1(\Gamma_g,e)
\quad \tto \quad \pi_0(\mathcal{H}omeo_g)
\end{equation}

induced from the monodromy of $X'_g$.
\end{defi}


To proceed we have to rely on the notion of morphism to an
orbifold, as described in Section \ref{sect-orbifund}, and the universality property of the (fine) moduli orbifold.





Let us recall from Section \ref{sect-trigostack} that, given a family of trigonal curves over a smooth projective variety $Y$, pointed at $y_0$, there is a natural $\Gamma_g$-torsor $P_Y$ over $Y$. The natural $\Gamma_g$-equivariant map from $P_Y$ to $X_g'$ is the datum which fixes the classifying morphism

\[
c: \quad Y \quad\tto\quad [X_g'/\Gamma_g].
\]


\begin{prop}\label{tauto}
Let $P_{y_0}$ be the fiber of $P_Y$ over $y_0\in Y$.
Given a family of trigonal curves over the smooth variety
$Y$, pointed at $y_0$, there is a commutative diagram

\[
\begin{matrix}
\pi_1(Y, y_0)  & \tto & \pi_0(\mathcal{H}omeo(P_{y_0})) \\[2mm]
\big\downarrow c_* & & \big\downarrow \;\cong \\[2mm]
\pi_1^{orb} ( [X_g'/ \Gamma_g], \Gamma_g \cdot x_0)
& \stackrel{\mu^{\mathcal{T}}_g}{\tto} & \pi_0(\mathcal{H}omeo_g)
\end{matrix}
\]
with horizontal monodromy maps and $c_*$ induced by the classifying
morphism $c$ from $Y$ to the moduli orbifold. 
\end{prop}

\proof
The family over $Y$ gives rise to a $\Gamma_g$-torsor $P_Y$ and
a classifying $\Gamma_g$-equivariant map $\tilde c: P_Y \to X_g'$
such that pull-back of the trigonal family from $Y$ and the pull-back
of the tautological trigonal family over $X'_g$ along $\tilde c$ are isomorphic.
Let $\tilde{y}\in P_{y_0}$, then the monodromy homomorphism associated to the family over $P_Y$ factors as

\[
\pi_1(P_Y, \tilde y) \quad \stackrel{\tilde c_*}{\tto} 
\quad \pi_1(X'_g,\tilde c (\tilde y))
\quad \tto \quad \pi_0 ( \mathcal{H}omeo_g)
\]

Then the diagram of the claim follows, since the monodromy is
trivial along $\Gamma$-orbits thanks to Prop.\ref{trivially}.
\qed

In order to show that the bottom map in Prop. \ref{tauto} is surjective
it suffices to construct a trigonal family with surjective monodromy. The rest of this section will be devoted to showcasing such a family.

\begin{defi} \rm
We define the \emph{tautological family} of branch data as the universal hypersurface $\mathcal{H}_{2g+4}\subset \PP^1 \times \PP^{2g+4}$
of degree $2g+4$ given by the homogeneous equation

\begin{equation}
\label{universal}
\sum_{i+j=2g+4}  a_i x_0^i x_1^j \quad = \quad 0.
\end{equation}
\end{defi}

\newcommand{\homeof}{\mathcal{H}omeo{\tilde{\:}\!}\!\;\!}

The hypersurface $\mathcal{H}_{2g+4}$ naturally defines a discriminant locus $\mathcal{D}_{2g+4}$, that is the locus of $a\in\PP^{2g+4}$ where at least two of the $2g+4$ roots of (\ref{universal}) coincide. Let $\mathcal{U}$ be the complement of $\mathcal{D}_{2g+4}$ 
inside $\PP^{2g+4}$. Consider now the trivial $\PP^1$-bundle $\PP^1\times \mathcal{U} \to \mathcal{U}$. Its structure group can be reduced from $\homeo_0$
to the subgroup $\homeo_{0,2g+4}$ of orientation preserving
homeomorphisms of a fibre, which preserve the intersection with 
$\mathcal{H}_{2g+4}$. Let $u_0\in \mathcal{U}$. Under the identification of  $\pi_1( \mathcal{U}, u_0)$ with the braid group of the sphere on $2g+4$ strands due to Zariski \cite{zariski-poincare}, the corresponding monodromy map is the
surjection (cf.\ \cite[(9.1), page 245]{fm})

\begin{equation}\label{mapclasssurj}
\xymatrix{\mu_{0,2g+4}:\pi_1( \mathcal{U}, u_0) \ar@{->>}[r] &
\pi_0(\homeo_{0,2g+4}) = \map_{0,2g+4}\\}.
\end{equation}

\begin{defi} \rm
Let $u_0\in \mathcal{U}$. Consider a simple generic triple cover $f:C\to \PP^1$ branched at the  $2g+4$ points of intersections with $\mathcal{H}_{2g+4}$ in the
fibre over $u_0$. The group of liftable homeomorphisms $\homeof_{0,2g+4}$ of $\PP^1$ with respect to $f$ is given by the 
homeomorphisms of $\PP^1$, such that there exists a homeomorphism of $C$ and a commuting diagram

\[
\xymatrix{ C \ar[r] \ar[d]_f & C \ar[d]^f \\
\PP^1 \ar[r] & \PP^1 \\}
\]


\end{defi}

Since, by simplicity, fibres over the branch points consist of only 2 points, any liftable homeomorphisms belongs to $\homeo_{0,2g+4}$. Thus, there is a sequence of successive inclusions

\begin{equation*}
\xymatrix{ \homeof_{0,2g+4} \ar@{^{(}->}[r] & \homeo_{0,2g+4}  \ar@{^{(}->}[r] & \homeo_{g}\\}.
\end{equation*}

\begin{thm}\label{surjmontrigo}
The monodromy map $\mu^{X'}_g$ 
descends to the surjective map

\begin{equation}
\xymatrix{ \mu_g^{\tfami}=
\pi_1^{orb}([X'_g/\Gamma_g], \Gamma_g\cdot x_0) \ar@{->>}[r] & \map_g.\\ }
\end{equation}

\end{thm}

\newcommand{\cM}{\mathcal{M}}
\newcommand{\cU}{\mathcal{U}}
\newcommand{\cH}{\mathcal{H}}

\begin{proof}
Let $\cM$ be the mapping class group associated to $\homeof_{0,2g+4}$. The group $\cM$ is of finite index inside $\map_{0,2g+4}$ \cite{bw}. 
Moreover, by means of the surjective map (\ref{mapclasssurj}), we can consider the inverse image of $\cM$ inside $\pi_1(\mathcal{U},u_0)$. This is also of finite index inside $\pi_1(\mathcal{U},u_0)$ and hence induces a finite \'etale cover $\widetilde{\mathcal{U}} \to \mathcal{U}$, which is a smooth complex variety. If we pull the trivial $\PP^1$-bundle over $\mathcal{U}$ back to $\widetilde{\mathcal{U}}$, then its structure group can be reduced to $\homeof_{0,2g+4}$ and its monodromy surjects onto $\cM$. Let $\tilde{\mu}_{0,2g+4}$ the monodromy map of $\tilde{\mathcal{U}}$ and $\tilde{u}_0$ a base point over $u_0$. The following diagram resumes the situation.
\medskip


\begin{equation}
\xymatrix{
& \pi_1(\widetilde{\cU},\tilde{u}_0) \ar[d] \ar[r]^-{\tilde{\mu}_{0,2g+4}}
& \pi_0(\homeof_{0,2g+4}) \ar[d] \ar@{=}[r]  & \cM\\
& \pi_1(\cU,u_0) \ar[r]^-{\mu_{0,2g+4}} & \pi_0(\homeo_{0,2g+4}) }
\end{equation}

In order to proceed we need to compare this with a second \'etale cover of $\mathcal{U}$. Let us describe it. Let $\cH^3_g$ be the Hurwitz space of simple triple covers of the
projective line by a smooth projective curve of genus $g$, considered up to
isomorphisms covering the identity on the projective line

\[
\begin{matrix}
C & \tto & C\\
\downarrow & & \downarrow \\
\PP^1 & \stackrel{Id}{\tto} & \PP^1
\end{matrix}
\]

By \cite[Thm. 1.53]{harris-morrison}
$\cH^3_g$ is an \'etale cover of $\cU$. Let us now compare the images of $\pi_1(\cH_g^3)$ and $\pi_1(\widetilde{U})$ inside $\pi_1(\cU)$. Composing the homomorphism of fundamental groups induced by the \'etale cover $\cH^3_g \to \cU$ with the surjective monodromy map $\mu_{0,2g+4}$, we get a monodromy map

\[
\mu_{0,2g+4}^{\cH^3}: \pi_1(\cH^3_g)\to \pi_0(\homeo_{0,2g+4}).
\]

By the second part  of \cite[thm. 1.53, page 33]{harris-morrison}
there exists a universal family of trigonal curves $\mathcal{C}^3_g \to \cH_g^3$ over the Hurwitz space. The existence of such a family implies that $\mu^{\cH^3}_{0,2g+4}$ factors through the homomorphism 

\[
\pi_0(\homeof_{0,2g+4})\to \pi_0(\homeo_{0,2g+4})
\]

induced by inclusion. Hence we get the following diagram (where for simplicity we omitted the base points)

\medskip

\qquad \qquad \qquad \xymatrix{
\pi_1(\cH^3_g) \ar@/_/[ddr] \ar@/^/[drr]
\ar@{.>}[dr]|-{\mu'} \\
& \pi_1(\widetilde{\cU}) \ar[d] \ar[r]^-{\tilde{\mu}_{0,2g+4}}
& \pi_0(\homeof_{0,2g+4}) \ar[d] \\
& \pi_1(\cU) \ar[r]^-{\mu_{0,2g+4}} & \pi_0(\homeo_{0,2g+4}) }

\medskip

We observe that $\pi_1(\widetilde{\cU})$ is the fiber product in the square. Hence by the universal property there exists a homomorphism

\[
\mu': \pi_1(\cH^3_g) \to \pi_1(\widetilde{\cU}).
\]
In fact, $\mu'$ is injective since it factors the 
injective map $\pi_1(\cH^3_g)\to \pi_1(\cU)$.
Thus there exists an \'etale cover $\cH^3_g\to \widetilde{\cU}$. We claim that this map is an isomorphism. In fact, suppose that we have two different covers $C\to \PP^1$ and $C'\to \PP^1$ inside $\cH^3_g$ that map to the same element in $\widetilde{\cU}$. This means that they have the same set of liftable homeomorphisms of $\PP^1$, hence they must have the same branch data. This in turn implies that they are the same cover.

\medskip

Moreover, every element of $\homeof_{0,2g+4}$ lifts uniquely to
$\homeo_g$, because the simple covering $f$ has no covering
transformations. 
Therefore the monodromy map of the family $\mathcal{C}^3_g$ of curves of genus $g$ factors as

\begin{equation}\label{monodromfact}
\xymatrix{\mu_{g}:\pi_1(\cH^3_g) \ar@{->>}^-{\tilde\mu\circ\mu'}[r] &
\pi_0(\homeof_{0,2g+4})\ar[r] & 
\pi_0(\homeo_{g})= \map_g},
\end{equation}

where we wrote $\tilde{\mu}$ to shorten $\tilde{\mu}_{0,2g+4}$. Finally, our surjectivity claim follows from the result of Hilden
\cite[Thm. 4, p.994]{hilden}, which states that the last two groups of the sequence are isomorphic.
To phrase it in the words of \cite[p. 24-25]{bw}: Hilden proved that, if $C\to S^2$ is a simple $3$-sheeted branched covering of a $2$-sphere, then every homeomorphism of $C$ is isotopic to a lifting of a homeomorphism of $S^2$.
\qed
\end{proof}

\begin{rmk}
Let $M$ be a Maroni invariant and $X_{M,g}$ the restriction of the vector bundle $X_g$ to the fixed Maroni invariant matrix locus $\Omega_{2,g+2}^M$. Let moreover $\mathcal{D}_{M,g}$ be the discriminant inside $X_{M,g}$ and $X'_{M,g}$ its complement. The corresponding Maroni stratum has a presentation as quotient stack $[X'_{M,g}/\Gamma_g]$.
\end{rmk}

\begin{rmk}
\label{surjmon}
By the upcoming Theorem \ref{isomaroni} in the next section,
the monodromy map $\mu^{\calx}_g$ and its restriction descend to surjective maps 

\begin{eqnarray*}
\mu_g^{\mathcal{T}}: \pi_1^{orb}([X'_g/\Gamma_g]) & \rightarrow & \map_g;\\
\mu_g^M: \pi_1^{orb}([X'_{M,g}/\Gamma_g]) & \rightarrow & \map_g.
\end{eqnarray*}

except for $\mu_g^{\maxi}$ for $g\equiv 1\ (\mod\ 3)$.

\end{rmk}

\section{Hirzebruch Surfaces, Discriminants and their topological invariants} 

Before some detailed arguments, let us state the result that we aim to show along the rest of this section.

The main goal of the present section is to prove the following theorem.

\begin{thm}
\label{isomaroni}
If $M<g/3-1$, the inclusion 
$[X'_{M,g}/\Gamma_g]\hookrightarrow [X'_g/\Gamma_g]$ induces a surjection

$$\pi_1^{orb}([X'_{M,g}/\Gamma_g]) \surj \pi_1^{orb}([X'_g/\Gamma_g])$$

with kernel a homomorphic image of $\ZZ/(M)$.
\end{thm}

It states that the orbifold fundamental groups $\pi_1^{orb}([\Xmg - \Dmg/ \Gamma_g])$ of the Maroni strata depend only very mildly on the Maroni invariant $M$, except in the cases excluded, when $M$ is maximal for its genus. In our argument, we will use a powerful theorem by Shimada (\cite[Cor 1.1]{shimada}, see Thm. \ref{thmshim} of this paper) that - under some hypotheses - puts the fundamental group of a fibration and the fundamental group of a single fiber into a short exact sequence. In other words, we can then exploit the fundamental group of the complement of the discriminant inside one single fiber of $X_g\to \Omega_{2,g+2}$ over a chosen matrix $\omega_0$ in $\omg$, that is with prescribed Maroni invariant.

As we have already stated in section \ref{sect-trigostack}, the fibers of $X_g$ parametrize the sections of the Tschirnhausen module that give rise to trigonal curves. Moreover, the projectivized vector bundle obtained from the Tschirnhausen module is a Hirzebruch surface $\mathbb{F}_M$. Hence, the projectivized space of each fiber of $X_g\to \Omega_{2,g+2}^M$ can be interpreted as a linear system $|T|$ on $\mathbb{F}_M$. Notably, if we consider trigonal curves of genus $g$ and Maroni invariant $M$ (which must meet the conditions $0\leq M \leq \frac13(g+2)$ and $M\equiv_2 g$), then $|T|$ corresponds to the linear system of type
$|3\sigma_0+c(g,M)f|$ on $\mathbb{F}_M$, where $\sigma_0$ is the movable section with $\sigma_0^2=M$, $f$ the ruling of $\mathbb{F}_M$ and $c(g,M)$ is the numerical function of $g$ and $M$ defined as $2c(g,M) = g + 2 - 3 M$. This in fact follows from the adjunction formula (here we denote the negative section by $\sigma_\infty:=\sigma_0 - M f$): 
\begin{eqnarray*}
2g-2 & = & C \cdot ( C + K) = (3 \sigma_0 + c(g,M)f) (2\sigma_0 - \sigma_\infty + (c(g,M)-2) f) \\
& = & 6 M + 2c(g,M) - c(g,M) + 3c(g,M) - 6,  
\end{eqnarray*}
hence $g - 1 = 3 M + 2c(g,M) - 3$. Our linear system $|T|$ has projective dimension $N:=2(m+n)+3$.

\newcommand{\Dz}{\mathcal{D}_0}
\newcommand{\fkd}{\mathfrak{D}}
\newcommand{\fkb}{\mathfrak{B}}

Consider now a fiber $F$ of $X_g$ over a matrix $\omega_0\in \Omega_{2,g+2}$, let $\cald_0$ be the restriction of the discriminant to $F$. Recall that $\PP(F)\cong|T|=\PP^N$. In fact, more precisely
the discriminant $\cald_0$ is exactly the cone over the projective dual variety $\mathbb{F}_M^* \subset \PP(F)$ of the Hirzebruch surface $\mathbb{F}_M \subset \PP(F)^*\cong |T|^*$. 
Let $\PP(\Dz)$ be its associated projectivized space inside $\PP(F)$. In order to understand the fundamental group of the complement of $\Dz$ it will be enough to consider the fundamental group of the complement of a generic plane section $\mathfrak{D}:= \PP^2 \cap \PP(\Dz)$, thanks to Zariski's theorem on the fundamental group of hyperplane sections of divisor complements \cite{zariski-poincare}. The dual operation of taking a plane section is projecting onto a plane and considering the branch divisor. Hence we project $\mathbb{F}_M$ onto a general plane $\PP^2\subset \PP(F)^*$, and consider the branch divisor $\mathfrak{B}$. This divisor is the dual curve of $\mathfrak{D}$ and it is not hard to compute its topological invariants. Once one assumes that $\PP(\Dz)$ has mild singularities, the 
curves $\fkb,\mathfrak D$ form a dual pair of Pl\"{u}cker curves. Via the Pl\"{u}cker formulas we get then the topological invariants of $\fkd$, which are needed to control the fundamental group of its complement. The claim will follow since the Pl\"{u}cker characteristics of $\fkb$ depend only on $g, K^2$ and $C\cdot C$ which are clearly independent of the Maroni invariant.

\subsection{Topological properties of the branch curve}

Let us start with the examination of the branch curve $\fkb$ of
a generic projection to $\PP^2$ of the image of $\mathbb{F}_M$ embedded in the projective space $\PP(F)^*\cong|T|^*$. 
The following claim is taken from \cite[Thm.1.1]{ciliberto-flamini}
and \cite[Prop.2.6]{cmt}.

\begin{prop}[\cite{ciliberto-flamini}]
\label{numerical}
Let $S$ be a smooth surface in $\PP^n$ of degree $d$
and $\rho:S\to \PP^2$ be a general projection.
Let $K$ and $H$ be the canonical and the hyperplane class
of $S$ and $K^2$, $e(S)$ its Chern numbers.
Then the branch curve of $\rho$ is an irreducible plane curve
with no singularities except for ordinary cusps and nodes
and with the following numerical characteristics:
\begin{enumerate}
\item
the degree of the branch curve is $b=3d+KH = 3H^2 + KH$,
\item
the number of ordinary nodes is
$e(S)-3K^2+24d+b(b-15)/2$,
\item
the number of ordinary cusps is
$2K^2-e(S)-15 d + 9b$.
\end{enumerate}
\end{prop}

\begin{cor}
\label{numerical2}
Let $S$ be a Hirzebruch surface embedded into $\PP(F)^*$
by the complete linear system $|T|$ of trigonal curves of
genus $g$, then the numerical characteristics of the branch curve of
a generic projection $\rho:S\to \PP^2$ depend only
on $g$.
\end{cor}

\proof
The Chern numbers are invariantly equal to $K^2=8$ and
$e(S)=4$. By adjunction, $2g-2 = H^2 + KH$, and a quick
calculation shows
\[
H^2 \quad = \quad (3\sigma_0+c(g,M) f)^2 
\quad = \quad 3 g + 6.
\] 
So the claim follows from the preceding proposition.
\qed

Assuming that $\mathfrak D,\fkb$ form a dual pair of Plücker
curves, we could deduce that also the generic plane section
$\mathfrak D$ is a Plücker curve with numerical characteristic
invariant but for their dependence on $g$.

Alas so much can certainly not be assumed in case 
$M=(g+2)/3$ since the linear system in not very ample in
this case. The positive result we can give is the following.

\begin{prop}
\label{plucker}
Let $\bff_M$ be a Hirzebruch surface embedded into $\PP(F)^*\cong |T|^*$
by the complete linear system of trigonal curves of
genus $g$. Then a generic plane section $\mathfrak D$
of its dual $\bff_M^*$ is a Plücker curve with numerical
characteristics only depending on $g$, if $c(g,M)\geq3$.
\end{prop}

\proof
We have to show that $\mathfrak D$ has no singularities
except for ordinary nodes and cusps. 
Let us stratify $\bff_M^*$ according to singularity type of generic plane sections, i.e.\  two points belong to the same stratum
iff generic plane sections with $\bff_M^*$ through them
produce topologically equivalent plane curve germs. 
The only singularity types which can thus be present
for $\mathfrak D$ are those belonging to strata in $\bff_M^*$
of codimension $1$.

Such strata are open in the set of hyperplanes which have a
degenerate singular intersection with $\bff_M$ or in the set of
hyperplanes which have at least two singular intersections.
The first set is irreducible under the given hypotheses
by a result of Shimada \cite[Prop. 4.9]{shimada-dual}.
%
%
In fact his criterion is that at each point $p$ the sections corresponding to our linear system and the fourth power $\mathfrak m^4_p$ of the local ideal generate the local algebra $\ofami_p$. The former can be identified with

\[
Span(x^iy^j\!,\: j\leq 3, i\leq c(g,M) +(3-j)M)
\quad (\mbox{resp. } \:
Span(x^iy^j\!,\: j\leq 3, i\leq c(g,M) +jM))
\]

in case of $p\not\in\sigma_\infty$
(resp.\ $p\in\sigma_\infty$), so the criterion applies since $c(g,M)\geq 3$. By the
irreducibility thus established, this part gives rise exclusively to ordinary
cusp singularities of $\mathfrak D$. So other irreducible open parts can only come from sets of hyperplanes having
intersections with $S$ which are smooth except for ordinary
nodes of which there are at least two.

Now, thanks to a result of Tyomkin  \cite[Prop. 2.11]{tyom}, we know that the sets of such hyperplanes not containing $\sigma_\infty$ are irreducible of 
codimension in $\PP(F)$ equal to the number of nodes.
Hence it suffices to prove that the set of hyperplanes
containing $\sigma_\infty$ is of codimension at least $3$.

In fact $\sigma_\infty$ is embedded in $\PP(F)^*$ as the
rational normal curve of degree $c(g,M)$. The linear
subspace of $\PP(F)^*$ spanned by this image must be contained
in any hyperplane which contains $\sigma_\infty$. So
the codimension of the family of such hyperplanes is
$c(g,M)+1$, one more than the dimension of the
linear subspace containing $\sigma_\infty$.
\qed

\subsection{An input from Shimada}

As we have already anticipated, a key role will be played by Cor.1.1 of the paper \cite{shimada} by Shimada. Since we will use this result thoroughly, it seems worth to recall it here.  The framework is the following. Let $f:A\to B$ be a dominant morphism from a smooth variety $A$ to a smooth variety $B$, with a connected general fiber. Moreover we assume that there exists a nonempty Zariski open subset $B^\circ \subset B$ such that $f$ is locally trivial in the $\mathcal{C}^\infty$ category over $B^\circ$. Let us now choose a base point $b\in B^\circ$, put $F_b:=f^{-1}(b)$ and choose a base point $\tilde{b}\in F_b$. Then the inclusion $\imath:F_b \hookrightarrow A$ induces a homomorphism $\imath_*:\pi_1(F_b,\tilde{B})\to \pi_1(A,\tilde{b})$. Here is the statement from \cite{shimada}.

\begin{thm}\label{thmshim}
Suppose that the following three conditions hold true:

\begin{enumerate}
\item the locus $Sing(f)$ of critical points of $f$ is of codimension greater than 2 in $A$;
\item there exists a Zariski closed subset $\varXi$ of $B$ of codimension $\geq 2$ such that $F_y:=f^{-1}(y)$ is nonempty and irreducible for any $y\in B\setminus\varXi$,
\item there exists a subspace $C\subset B$ containing $b$ and a continuous cross-section $s_C:C\to f^{-1}(C)$ of $f$ over $C$ satisfying $s_C(C)\cap Sing(f)=\emptyset$ and $s_C(b)=\tilde{b}$ such that the inclusion $C\hookrightarrow B$ induces a surjection $\pi_2(C,b) \to \pi_2(B,b)$.
\end{enumerate}

Let $i_{A_*}: \pi_1(A^{\circ},\tilde{b})\to \pi_1(A,\tilde{b})$ be the homomorphism induced by the inclusion $\imath_A:A^\circ\hookrightarrow A$. Then $Ker(\imath_*)$ is equal to 

$$\mathcal{R}:=\langle \{g^{-1}g^{\mu(\gamma)}|g\in \pi_1(F_b,\tilde{b}),\gamma\in Ker(i_{A_*})\}\rangle,$$

and we have the exact sequence

\begin{equation}\label{seqshimada}
1 \to \pi_1(F_b,\tilde{b})\big/\!\big/ Ker(i_{A_*})\stackrel{\imath_*}{\tto} \pi_1(A,\tilde{b})\stackrel{f_*}{\tto} \pi_1(B,b) \to 1.
\end{equation}
where the first group is the factor group by $\mathcal R$.
\end{thm}

We remark that the natural projection map $\Xmg ' \to \Omega_{2,g+2}$ is smooth and has no critical points. Hence the hypothesis of Thm \ref{thmshim} are true and we can apply it.

\subsection{Families with non-constant Maroni invariant}

Let $F_{M,g}$ be one fiber of the vector bundle $X_g\to\Omega_{2,g+2}$ with Maroni invariant $M$. We will denote by $F'_{M,g}$ the complement of the discriminant inside $F_{M,g}$. We recall that we denote by $\mini$ and $\maxi$ the minimal and the maximal Maroni invariant.

Since $\Omega_{\mini}$ is dense in $\Omega_{2,g+2}$, we can
find a disc $W\subset \Omega_{2,g+2}$ with center $O$ such that $W\cap \Omega_M=\{O\}$ and $W - \{O\} \subset \Omega_{\mini}$. Without loss of generality, shrinking $W$ if necessary, we have a fibration in generic (with respect to $\mathcal{D}$) $\CC^3$-subspaces over $W$ contained in $X_g$. This means that we get a flat family of cones over projective curves obtained by intersecting the discriminants with the $\CC^3$-fibration.  We obtain a diagram of inclusions

\begin{equation}\label{triviallocalsystem}
\xymatrix{ F_M'  \ar@{^{(}->}[r] & X_g' & \ar@{_{(}->}[l] F_{\mini}' \\
F'_M \cap \CC^3 \ar@{^{(}->}[u]\ar@{^{(}->}[r]  & X_g'\cap \CC^3\times W \ar@{^{(}->}[u]
&  F_{\mini}'\cap \CC^3 \ar@{_{(}->}[l] \ar@{^{(}->}[u] \\ }
\end{equation}
where $F'_{\mini}$ is the fiber of $X_g' \to \Omega_{2,g+2}$ over some $\omega_1 \neq O\in W\cap \Omega_{\mini}$ and $F'_M$ is the one over $\omega_0=O\in W$.

\begin{prop}
\label{minmaro}
If $g\neq4$
there is an isomorphism of fundamental groups

$$\pi_1(F'_{\mini}) \stackrel{\sim}{\tto} \pi_1(X'_g).$$

\end{prop}

\begin{proof}
The strategy is to apply Thm.\ \ref{thmshim} to the sequence of maps $F'_{\mini} \to X'_g \to \Omega_{2,g+2}$. In order to do this, we need to verify some properties of the projection $X'_g\to \Omega_{2,g+2}$. First, $\pi_1(\Omega_{2,g+2}) =\pi_2(\Omega_{2,g+2})=1$ since the complement of $\Omega_{2,g+2}$ in $\Mat_{g+2,g+4}$ has codimension at least 2 by Lemma \ref{purecod}.
Moreover, all fibers of $X'_g\to \Omega_{2,g+2}$ are irreducible, the map has no critical points and the minimal Maroni stratum $X'_{\mini}$ is contained in the part where the fibration is locally trivial in the differentiable category.
Thus we can apply Thm.\ \ref{thmshim} and we get
\[
1 \to \pi_1(F'_{\mini})\big/\!\big/ Ker(i_{X'_{g}*})\stackrel{\imath_*}{\tto} \pi_1(X'_g) \to 1.
\]
To get the claim we need to show that the subgroup $\mathcal R$ of Thm.\ref{thmshim} is trivial.
If $g$ is odd,  the minimal Maroni stratum $X'_{\mini} \subset X'_g$ has complement of codimension at least 2, since in this case on the minimal Maroni strata $n$ and $m$ are not equal. Hence it is clear that $Ker(i_{X'_{g}*})$ is trivial.

In case $g$ is even, $n=m$ on the minimal Maroni strata, hence the complement in $X'_g$ is codimension one. However, we can consider a diagram in the spirit of (\ref{triviallocalsystem}) with
fibrewise restriction to one dimension less:
\begin{equation}\label{triviallocalsystem2}
\xymatrix{ F_M'  \ar@{^{(}->}[r] & X_g' & \ar@{_{(}->}[l] F_{\mini}' \\
F'_M \cap \CC^2 \ar@{^{(}->}[u]\ar@{^{(}->}[r]  & X_g'\cap \CC^2\times W \ar@{^{(}->}[u]
&  F_{\mini}'\cap \CC^2 \ar@{_{(}->}[l] \ar@{^{(}->}[u] \\ }
\end{equation}
On the bottom row we get a fibration with each fibre a
complement of a cone over finitely many points.
In case $M=\mini-2$ and $g>4$ this fibration is trivial,
accordingly the action on fundamental groups is trivial.
We infer that the action of $Ker(i_{X'_{g}*})$ is trivial, so
$\mathcal R$ is trivial and we get our claim.
\qed
\end{proof}

The following proposition is needed in order to relate the fundamental group of a fiber with given Maroni invariant to the fundamental group of the total space.

\begin{prop}\label{genmaro}
Let $j_M:F'_{M,g} \hookrightarrow X'_g$ be the embedding of the complement of the discriminant inside a fiber $F_{M,g}$ with any given Maroni invariant $M$. The induced map $j_{M*}$ of fundamental groups is an isomorphism $\pi_1(F'_{M,g}) \stackrel{\sim}{\tto} \pi_1(X_g')$.
\end{prop}

\begin{proof}
The claim is true for the minimal Maroni stratum by Prop.\ \ref{minmaro}. Now let us consider once again diagram (\ref{triviallocalsystem}) and the $\pi_1$ of the spaces involved. We observe that the horizontal arrows in the bottom row induce isomorphisms thanks to Prop.\ref{plucker}, whereas the vertical arrows (left and right) induce isomorphisms by Zariski Theorem on generic sections \cite{zariski-poincare}. Finally, the map $\pi_1(F_{\mini}')\to \pi_1(X_g')$ is an isomorphism thanks to Prop.\ \ref{minmaro}, thus we conclude that all arrows need to induce isomorphisms. \qed
\end{proof}

\subsection{Some alternative quotient presentations}

Let us start with a few set-theoretical observations on the presentation of the  quotient stack.

We know we have a surjection $G \ra \Gamma_g$. We want to compare the subgroups in $G$ and in $\Gamma_g$ that stabilize matrices belonging to different strata of the moduli space. Let $x_0$ be a point in the fiber of the vector bundle $X_g\to \Omega_{2,g+2}$ over $\omega_0\in \Omega_{2,g+2}$. The point $\omega_0$ is naturally identified with a matrix. We can associate to $\omega_0$ the Maroni invariant $M$ since $\omega_0$ completely defines the splitting type of the Tschirnhausen module. As before, $\Xmg$ (respectively $\omg$) will denote the locus inside $X_g$ (resp.\ inside $\Omega_{2,g+2}$) that corresponds to that Maroni invariant. In the following, $G_{M,g}$ will be the stabilizer inside $G$ of $\omega_0\in \Omega_{2,g+2}^M$ and $\Gamma_{M,g}$ the respective stabilizer inside $\Gamma_g$. Let us denote by $F_{M,g}$ the fiber over $\omega_0$, that contains $x_0$. Then $G_{M,g}\neq \Gamma_{M,g}$ but their orbits inside $F_{M,g}$ are the same, since the kernel of the natural map $G_{M,g} \to \Gamma_{M,g}$ acts trivially. 
Since every $\Gamma_g$-orbit on $X'_{M,g}$ intersects with
$F'_{M,g}$ in a $\Gamma_{M,g}$-orbit we have set-theoretical bijections between

\begin{equation}\label{quotstacks}
[X'_{M,g}/G],\quad[X'_{M,g}/\Gamma_g],\quad[F'_{M,g}/\Gamma_{M,g}]\quad\text{and}\quad[F'_{M,g}/G_{M,g}].
\end{equation}

More is true in fact:

\begin{lemma}\label{altquot}
Let $F'_{M,g}\subset X'_{M,g}$ be as above and $\Gamma_*\subset \Gamma$ groups which act respectively on $F'_{M,g}$ and
$X'_{M,g}$. If the injection as a subspace induces a bijection between the respective orbits,
then it induces an isomorphism
\[
\pi_1^{orb}([F'_{M,g}/\Gamma_*] \quad\cong\quad
\pi_1^{orb}([X'_{M,g}/\Gamma])
\]
\end{lemma}

\begin{proof}
Let us use the shorthand notation $F$ and $X$ for $F'_{M,g}$
and $X'_{M,g}$.
The bijection of orbits is induced by the embedding $F\subset X$
and induces a homotopy equivalence from which we conclude
our claim:
\begin{eqnarray*}
F\times E\Gamma\big/\raisebox{-1mm}{$\Gamma_*$} \:\simeq\:
X\times E\Gamma\big/\raisebox{-1mm}{$\Gamma$}
& \implies &
\pi_1\left(F\times E\Gamma\big/\raisebox{-1mm}{$\Gamma_*$} \right)\:\cong\:
\pi_1\left(X\times E\Gamma\big/\raisebox{-1mm}{$\Gamma$}\right)\\
& \implies & \,\,\,\quad\pi_1^{orb}([F/\Gamma_*] \quad\cong\quad
\pi_1^{orb}([X/\Gamma])
\end{eqnarray*}
\qed\end{proof} 

\begin{prop}\label{isop1quot}

There is a sequence of isomorphisms

$$\pi_1^{orb}([X'_{M,g}/G])\cong\pi_1^{orb}([X'_{M,g}/\Gamma_g])\cong\pi_1^{orb}([F'_{M,g}/\Gamma_{M,g}])\cong\pi_1^{orb}( [F'_{M,g}/G_{M,g}]).$$

\end{prop}

\begin{proof}
In all quotients considered in (\ref{quotstacks}),
we remark that the acting group is connected, thus $\pi_0(\Gamma_g)=\pi_0(\Gamma_{M,g})=\pi_0(G_{M,g})=1$. The upshot is that the orbifold fundamental group of these spaces does not change if their presentations change. Two isomorphisms follow from Lemma \ref{altquot}, the last one from the following diagram:

\[\xymatrix{ \pi_2([X'_g/\Gamma_{M,g}]) \ar[r]  & \pi_1(\Gamma_{M,g}) \ar[r] & \pi_1(X'_g) \ar@{->>}[r] & \pi_1([X'_g/\Gamma_{M,g}]) \ar[r] & 1 \\
\pi_2([X'_g/G_{M,g}]) \ar[r]  & \pi_1(G_{M,g}) \ar@{->>}[u] \ar[r] & \pi_1(X'_g) \ar@{=}[u] \ar@{->>}[r] & \pi_1([X'_g/G_{M,g}]) \ar[u] \ar[r] & 1\\ }\]
In fact the $\pi_1$ of the quotients on the right are seen to be isomorphic by diagram chase. 
\qed\end{proof}

\medskip

The next aim of this subsection is to understand the natural map
\[
\pi_1(\Gamma_M) \to \pi_1(\Gamma_g),
\]
for $M$ as usual a Maroni invariant. 
The corresponding rank 2 vector bundle has a splitting type 
$(m,n)$, with $n-m=M$ and $n+m=g+2$. 
We identify the preimage of $\Gamma_M$ under the
map $G\to \Gamma_g$ with the group
$Aut(\ko(m)\oplus\ko(n))\times GL_2$.
Let us observe that under this identification the natural projection
of $G$ onto its last factor $GL_2$ gives
\begin{enumerate}
\item
the trivial map, when restricted to the first factor of
$Aut(\ko(m)\oplus\ko(n))\times GL_2$,
\item
an isomorphism, when restricted to the second factor
of $Aut(\ko(m)\oplus\ko(n))\times GL_2$, even though
this second factor is not identified with the last factor of $G$.
\end{enumerate}

Moreover, the chosen identification restricts to the map

\begin{eqnarray*}
Aut(\ko(m)) \times Aut (\ko(n)) & \to & Aut H^0(\ko(m-1)\oplus \ko(n-1))^*  \times  Aut H^0(\ko(m) \oplus \ko(n))\\ & & \cong GL_{n+m} \times GL_{n+m+2}; \\
&&\\
\lambda, \mu & \mapsto & \left( \begin{array}{cc} \lambda^{-1}Id_m & 0 \\ 0 & \mu^{-1}Id_n \end{array} \right),\ \left( \begin{array}{cc} \lambda Id_{m+1} & 0 \\ 0 & \mu Id_{m+1} \end{array} \right).
\end{eqnarray*}

Hence, we get a commutative (with non exact rows) diagram of groups


\begin{equation}\label{pi-coset1}
\xymatrix{
Aut(\ko(m)) \times Aut(\ko(n)) \ar@{^{(}->}[r] \ar@{^{(}->}[d]  & Aut(\ko(n)\times \ko(m))\times GL_2 \ar@{^{(}->}[d] \ar@{->>}[r] & \Gamma_M \ar[d]  \\
GL_{g+2} \times GL_{g+4} \ar@{^{(}->}[r] & GL_{g+2}\times GL_{g+4}\times GL_2 \ar[r] &  \Gamma_g \\
 }
\end{equation}

that in turn induces the following diagram of fundamental groups

\begin{equation}\label{groups}
\xymatrix{
\ZZ \times \ZZ \ar[r] \ar@{^{(}->}[d]_v  & \pi_1(Aut(\ko(n)\times \ko(m)))\times \ZZ \ar[d]_w \ar[r] & \pi_1(\Gamma_M) \ar[d]^{\rho}  \\
\ZZ \times \ZZ \ar[r] & \ZZ \times \ZZ \times \ZZ \ar[r] &  \pi_1(\Gamma_g),  \\
 }
\end{equation}

where the maps are named for further use.

\begin{lemma}\label{idealM}
The cokernel of the RHS vertical map $\rho: \pi_1(\Gamma_M) \to \pi_1(\Gamma)$ is a homomorphic image of $\ZZ\big/\raisebox{-1mm}{\!(M)}$.
\end{lemma}

\proof
Let us first notice that the vertical map $v: \ZZ \times \ZZ \to \ZZ \times \ZZ$ is given by the matrix $\left(\begin{smallmatrix} -m & -n \\ m+1 & n+1 
\end{smallmatrix} \right)$. Thus the image of the vertical map $w$ is generated by the columns of the matrix $\left(\begin{smallmatrix} -m & -n & * \\ m+1 & n+1 & * \\ 0 & 0 & 1 \end{smallmatrix}\right)$, where $*$ stands for possibly any integer value. The 1 in the bottom right entry is a consequence of our observation ii) here above.

Accordingly, the cokernel of $w$ is isomorphic to $\ZZ/(n-m)\cong \ZZ\big/\raisebox{-1mm}{\!(M)}$. It is now straightforward to see that we can plug the RHS commutative square from Diagram \ref{groups} into the following commutative diagram, with exact rows and columns.



\begin{equation}\label{pi-coset2}
\xymatrix{
1 \ar[r] & \pi_1(\CC^*) \ar@{=}[d] \ar[r] & \pi_1(G_M) \ar@{^{(}->}[d] \ar[r] & \pi_1 (\Gamma_M) \ar[d] \ar[r] & 1 \ar@{=}[d] \\
1 \ar[r] & \pi_1(\CC^*) \ar[r] & \pi_1(G) \ar@{->>}[d] \ar[r] & \pi_1 (\Gamma) \ar@{->>}[d] \ar[r] & 1 \\
 & & \ZZ\big/\raisebox{-1mm}{\!(M)} \ar@{->>}[r] & coker(\rho) \\ 
 }
\end{equation}
\qed

\subsection{Proof of Theorem \ref{isomaroni}}

Now we are ready to give a proof of our main Thm. \ref{isomaroni}.

\medskip

\proof
We have the following commutative diagram.

\begin{equation}
\xymatrix{\pi_1 (\Gamma_{M,g}) \ar[r] \ar[d] & \pi_1 (F_{M,g}) \ar[d] \ar[r] & \pi_1^{orb} ([F_{M,g}/\Gamma_{M,g}] )\ar[r]\ar[d] & 1\\
\pi_1 (\Gamma_g) \ar@{=}[d] \ar[r] & \pi_1 (X_{M,g}') \ar[d] \ar[r] & \pi_1^{orb} ([X'_{M,g}/\Gamma_g]) \ar[d] \ar[r] & 1 \\
\pi_1 (\Gamma_g) \ar[r] & \pi_1 (X_g') \ar[r] & \pi_1^{orb} ([X_g'/\Gamma_g]) \ar[r] & 1 \\ }
\end{equation}

The exact rows stem from the long exact homotopy sequence
associated to the respective group actions. The commutativity
of the squares on the left follows from the commutativity of
the underlying continuous maps, while on the right the maps
between the orbifold fundamental groups are defined in exactly
the way to make the diagram commutative.

Next we add more information to the diagram. Let us note that by
Prop.\ \ref{genmaro} the composition of the vertical maps
in the middle is an isomorphism, hence the first factor is injective
the second is surjective. Hence in the bottom right square all maps are surjective. Finally the map at the right top
of the diagram is an isomorphism thanks to Lemma \ref{altquot}.

Discarding the middle row and using the isomorphism at the
right top we get a new commutative diagram, where the
central map is an isomorphism by Prop.\ \ref{genmaro} and 
thus can be transversed in both directions.
\begin{equation}
\xymatrix{\pi_1 (\Gamma_{M,g}) \ar[r] \ar[d] & \pi_1 (F_{M,g}) \ar@{<->}^{\cong}[d] \ar@{->>}[r] & \pi_1^{orb} ([X'_{M,g}/\Gamma_g]) \ar@{->>}[d] \ar[r] & 1 \\
\pi_1 (\Gamma_g) \ar[r] & \pi_1 (X_g') \ar@{->>}[r] & \pi_1^{orb} ([X_g'/\Gamma_g]) \ar[r] & 1 \\ }
\end{equation}
Thus the kernel we are interested in is the homomorphic image
of $\pi_1(\Gamma_g)$, since all its elements come from
elements in the kernel of the surjection in the bottom row.

Moreover the image of $\pi_1(\Gamma_{M,g})$ in $\pi_1(\Gamma_g)$ maps to the zero of 
$\pi_1^{orb} ([X'_{M,g}/\Gamma_g])$.
We conclude with the help of Lemma \ref{idealM} that our kernel
is the homomorphic image of $\ZZ/(M)$.
\qed

\section{Presentation of the fundamental group}\label{sect-presentation}

The goal of the present section is to give a complete presentation of the orbifold fundamental group of the Maroni stratum in case 
$3M=g+2$. This is the case where the upper bound for the
Maroni invariant $\frac13(g+2)$ is attained. Our main theorem is the following.

\begin{thm}\label{presentgroup}
The orbifold fundamental group 
$\pi_1^{orb}[X'_{\maxi,g}/\Gamma_g])$ of the maximal Maroni stratum, in the case $g \equiv 1\ (\mod\ 3)$ has a presentation in terms of generators $t_1,\dots, t_{2g+2}$, and relations:

\begin{itemize}
\item of "diagram type"
\begin{eqnarray*}
t_it_jt_i = t_jt_it_j & \text{if }\ j=i+1\ \text{ or }\ j=i+2;\\
t_it_j=t_jt_i & otherwise, 
\end{eqnarray*}

and

\begin{equation*}
(t_it_jt_i^{-1})t_k=t_k(t_it_jt_i^{-1})\ \text{ if }\ i+1=j=k-1. 
\end{equation*}




\item of "global type":
denote $\delta_0=t_1t_2\: t_3t_4 \cdots t_{2g+1}t_{2g+2}$

\begin{eqnarray*}
\delta_0 & \text{centralizes} & t_{2g+1} t_{2g-1} \cdots t_3 t_1, \\
&& t_{2g+2} t_{2g} \cdots t_4 t_2.
\end{eqnarray*}

\item of "quotient type".

let us denote $\delta_1=t_{2g+1}t_{2g+2}\: 
t_{2g-1} t_{2g} \cdots t_1t_2.$

\begin{eqnarray*}
(\delta_0\delta_1)^3 & = & 1\\
\delta_0^{g+2} &= &1. 
\end{eqnarray*}

\end{itemize}
\end{thm}

Here the relations of "diagram type" are encoded by the graph of
figure \ref{graphfig}.
\unitlength=1.6mm
\begin{figure}[h]
\begin{picture}(30,15)(-30,-3)

\put(0,0){\circle*{.3}}
\put(0,10){\circle*{.3}}
\put(10,0){\circle*{.3}}
\put(10,10){\circle*{.3}}
\put(20,0){\circle*{.3}}
\put(20,10){\circle*{.3}}
\put(30,0){\circle*{.3}}
\put(30,10){\circle*{.3}}
\put(40,0){\circle*{.3}}
\put(40,10){\circle*{.3}}

\put(1.5,1.5){\line(1,1){7}}
\put(11.5,1.5){\line(1,1){7}}
\put(31.5,1.5){\line(1,1){7}}

\put(1.5,0){\line(1,0){7}}
\put(1.5,10){\line(1,0){7}}
\put(11.5,0){\line(1,0){7}}
\put(11.5,10){\line(1,0){7}}
\put(31.5,0){\line(1,0){7}}
\put(31.5,10){\line(1,0){7}}

\put(0,1.5){\line(0,1){7}}
\put(10,1.5){\line(0,1){7}}
\put(20,1.5){\line(0,1){7}}
\put(30,1.5){\line(0,1){7}}
\put(40,1.5){\line(0,1){7}}

\put(22,5){\circle*{.5}}
\put(25,5){\circle*{.5}}
\put(28,5){\circle*{.5}}

\put(-1,-3){$2$}
\put(9,-3){$4$}
\put(19,-3){$\cdots$}
\put(29,-3){$\cdots$}
\put(39,-3){$2g+2$}

\put(-1,12){$1$}
\put(9,12){$3$}
\put(17,12){$\cdots$}
\put(29,12){$\cdots$}
\put(39,12){$2g+1$}

\end{picture}
\caption{}
\label{graphfig}
\end{figure}


\subsection{Comparison with Weierstrass parameter space}\label{compweier}

\renewcommand{\max}{M}
\newcommand{\Tgm}{\mathcal T_{g,m}}
\newcommand{\Aut}{\operatorname{Aut}}
\newcommand{\Hom}{\operatorname{Hom}}
\newcommand{\Cstartimes}{\CC^{^{\scriptstyle*}}\!\!\times\!}
\newcommand{\ibold}{i}
\newcommand{\jbold}{j}
\newcommand{\kbold}{k}


We are going to give a presentation of the orbifold fundamental
group of the trigonal stratum $[X'_{\max}/\Gamma_g]$ of Maroni invariant $\max$ in the
moduli space of curves of genus $g= 3\max -2$.

First we give an identification with the orbifold fundamental
group of another quotient.
To this end we give a concrete description of the linear system
$\PP V_\max$ associated to the divisor $3\s_0$ on $\FF_\max$.
Consider the isomorphism
\[
\begin{matrix}
\CC \:\times&\!\!\!\! \CC[x_1,x_0]_{\max} \:\times&\!\!\!\!   \CC[x_1,x_0]_{2\max} \:\times&\!\!\!\!   \CC[x_1,x_0]_{3\max}
& \, \tto \, & V_\max = H^0(\FF_\max,\ko_{\FF_\max}(3\s_0)) \\
u_0 ,& u_1(x_1,x_0) ,& u_2(x_1,x_0) ,& u_3(x_1,x_0) & \mapsto & 
u_0 y^3 + u_1 y^2 + u_2 y + u_3
\end{matrix}
\]

which uses the homogeneous coordinates $x_1,x_0$ 
on the base $\PP^1$ and the inhomogeneous coordinate $y$
on the fibre. 
The left hand side is acted on by $\CC^*\times\CC^*\times GL_2$ 
where

\begin{enumerate}
\item
the first factor $\CC^*$ acts diagonally by homotheties,
\item
the second factor $\CC^*$ acts on the polynomials 
$u_{\nu}$ by $\lambda^\nu$,
\item 
$GL_2$ acts on the coordinate vector $(x_1,x_0)$.
\end{enumerate}
The set of sections with singular zero-locus is preserved under
this action.

\begin{prop}
\label{linsysquo}
Let $\dfami_{V_\max} \subset \PP V_\max$ be the discriminant corresponding
to singular divisors, then

\[
\pi_1^{orb} ([X'_{M}/\Gamma_g]) \quad = \quad 
\pi_1^{orb} \big( (\PP V_\max \setminus \dfami_{V_\max} ) \big/
\raisebox{-1mm}{$\Cstartimes GL_2$} \big)
\]

with respect to the action of $\{ 1 \} \times \CC^*\times GL_2$
induced on the linear system.
\end{prop}

\proof
From Prop.\ref{isop1quot} the orbifold fundamental group on
the left hand side is isomorphic to
\[
\pi_1^{orb} ([F'_\max/\Gamma_\max]) .
\]
On the other hand $F_\max$ is identified with $H^0(\FF_\max, \ofami_{\FF_\max}(3\s_0))=V_\max$. Of course, also the discriminants are identified so we infer $\PP V'_\max = \PP F_\max'$. To take the group actions into account we notice the following three facts,

\begin{enumerate}
\item
the action of $\Gamma_\max$ is induced by the action of $G_\max$ on
$\FF_M$, which in turn is induced by the action of

\medskip

\begin{equation*}
\Aut \ofami(\max) \times \Aut\ofami (2\max) \rtimes 
\Hom(\ofami(\max),\ofami(2\max))\times GL_2
\end{equation*}

on the corresponding rank $2$ bundle
$\ofami(\max) \oplus\ofami (2\max)$ over $\PP^1$.

\item
under the identification $\PP V'_\max = \PP F_\max'$ the action
of $\CC^*\times\CC^*\times GL_2$ factors through the isomorphism
\begin{equation}
\label{autiso}
\CC^*\times\CC^*\times GL_2
\quad \cong \quad
\Aut \ofami(\max) \times \Aut\ofami (2\max) \times GL_2.
\end{equation}

\item
Under the identification $V_\max= H^0 \ofami
+ H^0 \ofami(\max) + H^0 \ofami(2\max) + H^0 \ofami(3\max)$
the action induces the map

\[
\begin{array}{c@{}c@{}ccc@{}c@{}c@{}c@{}c@{}c@{}c}
\Aut \ofami(\max) &\times &\Aut\ofami (2\max)
& \tto  &
\Aut \ofami &\times &\Aut\ofami (\max) &\times &\Aut\ofami (2\max) 
&\times &\Aut\ofami (3\max)
\\
\lambda, &&\eta & \mapsto & 
\lambda^2\eta\inv\negthickspace\negthickspace,
&& \lambda, &&\eta, &&\lambda\inv \eta^2
\end{array}
\]

\end{enumerate}

We can therefore argue with the following commutative diagram 
\begin{equation}\label{loops}
\xymatrix{
\pi_1 (\Cstartimes\Cstartimes GL_2) \ar[r] \ar[r]\ar^{\sim}[d] & 
\pi_1 (\PP V_\max') \ar@{=}[d] \ar[r] & 
\pi_1^{orb} ([\PP V_\max'/\Cstartimes GL_2]) \ar[r]\ar^{\sim}[d] & 1
\\
\pi_1 (\Gamma_\max)  \ar[r] & 
\pi_1 (\PP F_\max')  \ar[r] & 
\pi_1^{orb} ([F'_\max/\Gamma_\max])  \ar[r] & 1 
}
\end{equation}

First we note that the top row is exact, even though we divide out
by a subgroup of the group on the left hand side. But this does not
matter since the additional factor acts transitively on the fibres
of $V_\max' \to \PP V_\max'$.

Similarly the bottom row is exact since the quotient map
$F_\max' \to \PP F_\max'$ is obtained by a free action of a subgroup
of $\Gamma_\max$. The first two vertical maps then follow from 
equation (\ref{autiso}) and the identity
$\PP V'_\max = \PP F_\max'$. 
With the five-lemma we get the final isomorphism to complete
our proof.
\qed  

Very much in the spirit of \cite{loenne-duke} a presentation of 
the knot group of $\dfami_{V_\max}\subset \PP V_\max$
has been obtained: 

\begin{thm}[\cite{loenne-weierstrass}, thm.2]
\label{theorem2}
Let $\PP V_\max'$ be the discriminant complement in the linear 
system $|3 \sigma_0|$ of trigonal curves of genus $g$ on the ruled surface 
$\PP(\ofami(\max)\oplus\ofami(2\max))$ with $3\max=g+2$.
Then $\pi_1(\PP V_\max')$ is generated by elements

$$ 
T_1, \dots , T_{2g+2}
$$

with a complete set of relations provided in terms
of the edges $E_\max$ of the graph below:
\\

\unitlength=1.6mm
\begin{figure}[h]
\begin{picture}(30,15)(-30,-3)

\put(0,0){\circle*{.3}}
\put(0,10){\circle*{.3}}
\put(10,0){\circle*{.3}}
\put(10,10){\circle*{.3}}
\put(20,0){\circle*{.3}}
\put(20,10){\circle*{.3}}
\put(30,0){\circle*{.3}}
\put(30,10){\circle*{.3}}
\put(40,0){\circle*{.3}}
\put(40,10){\circle*{.3}}

\put(1.5,1.5){\line(1,1){7}}
\put(11.5,1.5){\line(1,1){7}}
\put(31.5,1.5){\line(1,1){7}}

\put(1.5,0){\line(1,0){7}}
\put(1.5,10){\line(1,0){7}}
\put(11.5,0){\line(1,0){7}}
\put(11.5,10){\line(1,0){7}}
\put(31.5,0){\line(1,0){7}}
\put(31.5,10){\line(1,0){7}}

\put(0,1.5){\line(0,1){7}}
\put(10,1.5){\line(0,1){7}}
\put(20,1.5){\line(0,1){7}}
\put(30,1.5){\line(0,1){7}}
\put(40,1.5){\line(0,1){7}}

\put(22,5){\circle*{.5}}
\put(25,5){\circle*{.5}}
\put(28,5){\circle*{.5}}

\put(-1,-3){$1$}
\put(9,-3){$2$}
\put(19,-3){$\cdots$}
\put(29,-3){$\cdots$}
\put(39,-3){$g+1$}

\put(-5,12){$g+2$}
\put(7,12){$g+3$}
\put(17,12){$\cdots$}
\put(29,12){$\cdots$}
\put(39,12){$2g+2$}

\end{picture}
\caption{}
\label{graph}
\end{figure}

\begin{enumerate}
\item
$ T_\ibold T_\jbold= T_\jbold T_\ibold$
for all $(\ibold,\jbold)\not\in E_{\max}$,
\item
$ T_\ibold T_\jbold T_\ibold= T_\jbold T_\ibold T_\jbold$
for all $(\ibold,\jbold)\in E_{\max}$,
\item
$ T_\ibold T_\jbold T_\kbold T_\ibold= T_\jbold T_\kbold T_\ibold T_\jbold$
for $\ibold<\jbold<\kbold$
such that
$(\ibold,\jbold),(\ibold,\kbold),(\jbold,\kbold)\in E_{\max}$,
\item
for all $\jbold$
\\[-2mm]
\[
\bigg( T_\jbold\inv \Big( T_{2g+2} T_{2g+1} \cdots T_2 T_1
\Big)\bigg)^{g+1}
\,=\quad 
\bigg(\Big( T_{2g+2} T_{2g+1} \cdots T_2 T_1
\Big) T_\jbold\inv \bigg)^{g+1}
\]
\end{enumerate}
\end{thm}

\paragraph{Remark}
The claim is changed from the given source in so far as we restrict
to the case $n=1$ and give the indices in the usual linear order,
$1$ replaces $(1,1)$, $2$ replaces $(1,2)$ and so on till $2g+2$ replaces $(2,g+1)$.
\\

We need to take a little bit more from that paper:

\begin{prop}
\label{freehomotopic}
Let $\delta_0$ and $\delta_1$ be the homotopy classes
given by
\begin{equation}
\label{products}
\delta_0 = T_{2g+2} T_{2g+1} \cdots T_2 T_1
,\qquad
\delta_1 = T_{g+2} T_1\: T_{g+3} T_2
\:\cdots\: T_{2g+3} T_{g+1},
\end{equation}
then the following paths represent their \emph{free} homotopy
classes:
\[
y^3 + x^{g+2} + e^{2\pi i t}, \qquad
y^3 + e^{2\pi i t} x^{g+2} + 1.
\]
\end{prop}

In fact we may deduce, with the commutation relations from
$i)$ above, that
\begin{equation}
\label{reorder}
\delta_0 = T_{2g+2} T_{2g+1} \cdots T_2 T_1
\quad = \quad
T_{2g+2} T_{g+1} \cdots 
T_{g+3} T_2\:T_{g+2} T_1
\end{equation}

To get a concise presentation of the given group, it is possible
to exploit the following isomorphism of free groups of free rank 
$2g+2=6\max-2$.
\[
\begin{array}{ccclcl}
t_i & = & \s_{j} & j= 6\max - ( i + 3 )/2 & \text{if} & i \equiv_6 1
\\
t_i & = & \s_{j} & j = 3\max - i/2  & \text{if} & i \equiv_6 2
\\
t_i & = & \s_{j} & j = 3\max-(i+1)/2 & \text{if} & i \equiv_6 3
\\
t_i & = & t_{i-1}\inv\s_j t_{i-1} & j= 6\max - ( i + 2 )/2 
& \text{if} & i \equiv_6 4
\\
t_i & = & t_{i+1}\s_j t_{i+1}\inv & j= 3\max - ( i + 1 )/2
& \text{if} & i \equiv_6 5
\\
t_i & = & \s_{j} & j = 6\max - (i+2)/2  & \text{if} & i \equiv_6 0
\end{array}
\]


Thanks to $t_{2i-1} t_{2i} = T_{2j}T_{j}$ for
$j = 3M - i$
the special elements are given in the new generators as
\begin{eqnarray*}
\delta_0 & = & t_1 t_2\: t_3 t_4\:\cdots\: t_{2g+1} t_{2g+2}
\\
\delta_1 & = & t_{2g+1} t_{2g+2}\: t_{2g-1} t_{2g} \: \cdots \: t_{1}t_{2}
\end{eqnarray*}

This isomorphism lies at the heart of the
following corollary.

\begin{cor}
\label{recast}
The fundamental group $\pi_1(\PP V_\max')$ has a finite presentation 
in terms of generators $t_1,\dots, t_{2g+2}$, and relations\\
$i)$ of "diagram type"
\begin{eqnarray*}
t_it_jt_i = t_jt_it_j & \text{ if }\ j=i+1,i+2;\\
t_it_j=t_jt_i & \text{otherwise}, 
\end{eqnarray*}
and
\begin{equation*}
(t_it_jt_i^{-1})t_k=t_k(t_it_jt_i^{-1})\ \text{when}\ i+1=j=k-1. \\
\end{equation*}
$ii)$ of "global type"
\begin{eqnarray*}
\label{asymp}
\delta_0 & \text{centralizes} & t_{2g+1} t_{2g-1} \cdots t_3 t_1, \\
&& t_{2g+2} t_{2g} \cdots t_4 t_2. \notag
\end{eqnarray*}
where we denote $\delta_0=t_1 \cdots t_{2g+2}$.
\end{cor}

\paragraph{Aside} It needs quite an effort to get all the relations
of part $iv)$ of the theorem essentially out of the two relations
of (\ref{asymp}).
\\

\subsection{The group action and free loops}\label{groupaction}

With Prop.\ref{linsysquo} in mind we want to understand the
left hand map of the exact sequence $(\ref{loops})$.
\[
\pi_1 (\Cstartimes\Cstartimes GL_2) \quad \tto \quad
\pi_1 (\PP V_\max') 
\]
Elements in the domain can be represented by loops
based at the identity which are real $1$-parameter subgroups. An image is then represented by the corresponding orbit of a chosen basepoint. 

We choose four $S^1$-subgroups of $\Cstartimes\Cstartimes GL_2$, each parametrized by complex numbers $\lambda$ of unit length:
\[
(\lambda, \lambda, \big(\begin{smallmatrix} 1&0\\0 & 1 \end{smallmatrix}\big) ),\quad
(\lambda, \lambda^2, \big(\begin{smallmatrix} 1&0\\0 & 1 \end{smallmatrix}\big) ),\quad
(1,1, \big(\begin{smallmatrix} \lambda&0\\0 & 1 \end{smallmatrix}\big) ),\quad
(1,1, \big(\begin{smallmatrix} 1&0\\0 & \lambda\end{smallmatrix}\big) ).
\]

As we noted before, the first subgroup acts trivially on the projective space $\PP V_\max$, since it coincides with the circle action on the coordinates $y$ and $y_0$.
\begin{eqnarray*}
y & \mapsto & \lambda y \\
y_0 & \mapsto & \lambda y_0
\end{eqnarray*}

The action of the second subgroup on coordinates is 
\begin{eqnarray*}
y_0& \mapsto & \lambda y_0 \\
x,x_0,y & \mapsto & x,x_0,y
\end{eqnarray*}
The remaining two act by multiplication by $\lambda$ solely on $x$, respectively $x_0$.

In $\PP V_\max$ let us take now the point represented by $y^3+x^{2g+2} + x_0^{2g+2}$. We neglect the first constant orbit, the other three orbits are represented by
\[
y^3+\lambda^3x^{2g+2} + \lambda^3x_0^{2g+2},\quad
y^3+\lambda^{2g+2}x^{2g+2} + x_0^{2g+2},\quad
y^3+x^{2g+2} + \lambda^{2g+2}x_0^{2g+2}.
\]

With these pieces of information we can obtain the image classes in $\pi_1(\PP V_\max')$ of the three essential subgroups.

\begin{lemma}
\label{homclasses}
The images of the three essential subgroups are respectively
\[
\delta_0^3\delta_1^3,\quad
\delta_0^{2g+2}\!, \quad \text{and}\quad
\delta_1^{2g+2}.
\]
\end{lemma}

\begin{proof}
By Prop.\ref{freehomotopic} the last orbit is freely homotopic
to $\delta_1^{2g+2}$, hence the claim is true up to conjugation.
It is in fact true, since $\delta_1^{2g+2}$ can be shown to be central in $\pi_1(\PP V_\max')$. The same argument applies
to the second orbit.

Now let us go back to $y^3+x^{2g+2}+x_0^{2g+2}$. We use it again as a base point but look at the family
\[
y^3+ex^{2g+2}+e_0x_0^{2g+2}.
\]
All members define smooth trigonal curves except for $e=0$ or $e_0=0$. Let us denote by $E=\{e,e_0\}$ the parameter space of this family. The discriminant is a normal crossing divisor in $\CC^2$, the union of the axes. This means that $E'\cong \CC^*\times \CC^*$ and $\pi_1(E')\cong \ZZ \times \ZZ$.  We can thus conclude that $\delta_0$ and $\delta_1$ commute.

Moreover the first orbit is now identified with $\delta_0^3 \delta_1^3$ up to conjugacy. Also in this last case the claim follows, since this element can be shown to be central, too.
\qed
\end{proof}

In this way the argument of the section is almost completed. 
It only remains to combine Cor.\ref{recast} and add the elements in the claim of Lemma \ref{homclasses} to get the claim of
Theorem \ref{presentgroup}. 
This procedure is justified by Prop.\ref{linsysquo}.
As a last improvement we note that the last two elements in
Lemma \ref{homclasses} coincide.

\subsection{The symplectic action on the first homology group}

Let us conclude with the implication to the monodromy map in the case $M=\frac{g+2}{3}$. Instead of the mapping class group of genus $g$ itself we consider the representation on the first homology group with $\ZZ/2\ZZ$ coefficients of the complex curve $C_g$ of genus $g$. We will denote by $\underline{c}$ the class in $H_1(C_g,\ZZ_2)$ of a closed (real) curve $c\subset C_g$. Let $t_i$ be the generators of the orbifold fundamental group as in Section \ref{compweier}. Since they map to Dehn twists along simple closed curves $c_i$ on $C_g$, the induced maps on the first homology are Picard-Lefchetz transvections

\begin{eqnarray*}
\tau_i: H_1(C_g,\ZZ_2) & \to & H_1(C_g,\ZZ_2);\\
\underline{c} & \mapsto & \underline{c} + \langle c, c_i \rangle \underline{c_i}
\end{eqnarray*}

where $\langle c, c_i \rangle$ is the parity of the transversal intersection of $c_i$ with a representative $c$ of the class $\underline{c}$.

\begin{lemma}\label{transvect}
The generators $t_1, \dots , t_{2g+2}$ map to transvections $\tau_1, \dots , \tau_{2g+2} \in Sp H_1 (\Sigma_g, \ZZ_2)$ on elements $c_1, \dots, c_{2g+2} \in H_1(\Sigma_g,\ZZ_2)$ on which the intersection pairing is given by the diagram below.
\end{lemma}

\unitlength=1.9mm
\begin{picture}(15,18)
\put(2,3){\includegraphics[scale=.25]{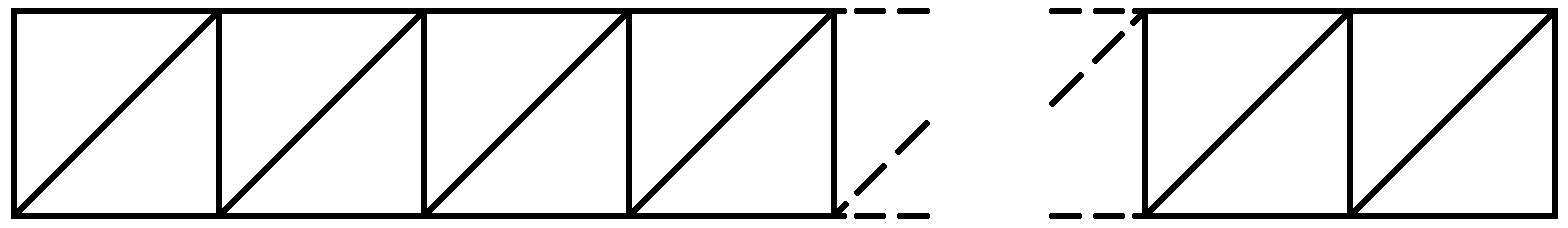}}
\end{picture}

\newcommand{\Fm}{F_{Mil}}

\proof
If we consider $y^3+x^{g+2}$ as a singular curve in the trigonal linear system, then the Milnor fibre of its isolated singularity is naturally identified with the intersection of a small ball with any sufficiently close smooth curve $C_g$ in our trigonal linear system.
In fact, the elements $c_i$ are linearly independent in the $\ZZ/2\ZZ$ first homology of the Milnor fiber $\Fm$ of the function $y^3+x^{g+2}$.
It is then readily checked that for $i\neq j$
\begin{enumerate}
\item
$\tau_i \tau_j = \tau_j\tau_i$ implies $\langle c_i,c_j\rangle =0$,
\item
$\tau_i \tau_j\tau_i = \tau_j\tau_i\tau_j$ implies $\langle c_i,c_j\rangle =1$.
\end{enumerate}
Hence the intersection diagram of the $c_i$ coincides with the diagram of figure \ref{graphfig} encoding the commutation and braid relations of pairs of $t_i$.
\qed

By our choice of the generators $t_i$, the diagram
gives generators $c_i$ of the $\ZZ/2\ZZ$ first homology of the Milnor fiber $\Fm$ of the function $y^3+x^{g+2}$, as computed by Pham \cite{ph} and Hefez-Lazzeri \cite{hl}.

The isolated singularity given by  $y^3+x^{g+2}$ is of type $J_{M,0}$ according to the second table of \cite[p.248]{avgs}, where again $M = \frac{g+2}3$. According to \cite[Tabelle 3, p.484]{eb-quad-form} its integral intersection lattice is isomorphic to
\begin{eqnarray*}
& \displaystyle \bigoplus^{l} E_8 \oplus \bigoplus^{2l} 
\big(\begin{smallmatrix} 0 & 1 \\ 1 & 0
\end{smallmatrix} \big) \oplus D_4
&\text{if } M=2l+1 \text{ is odd},
\\
\text{respectively}
& \displaystyle \bigoplus^{l} E_8 \oplus \bigoplus^{2l-2} 
\big(\begin{smallmatrix} 0 & 1 \\ 1 & 0
\end{smallmatrix} \big) \oplus 0 \oplus 0
&\text{if } M=2l \text{ is even}.
\end{eqnarray*}

After reduction$\pmod 2$ the corresponding $\ZZ_2$-vector space of dimension $2g+2$
is seen to have a radical of rank $2$, coming from the 
reduction of the summand $D_4$ in case $M$ is odd.
This radical is generated by the elements
\begin{eqnarray*}
&&
c_1+ c_4 \: \:+ c_7 + c_{10}  \: \:+ \cdots + c_{2g-1} + c_{2g+2},\\
&&
c_1+c_2+c_3\: \:+ c_7+c_8+c_9 \: \:+\cdots 
+ c_{2g-1}+ c_{2g}+c_{2g+1}.
\end{eqnarray*}
The support of the corresponding elements can be given on the
Dynkin-diagram as

\unitlength=1.9mm
\begin{picture}(15,20)
\put(10,12){\includegraphics[scale=.16]{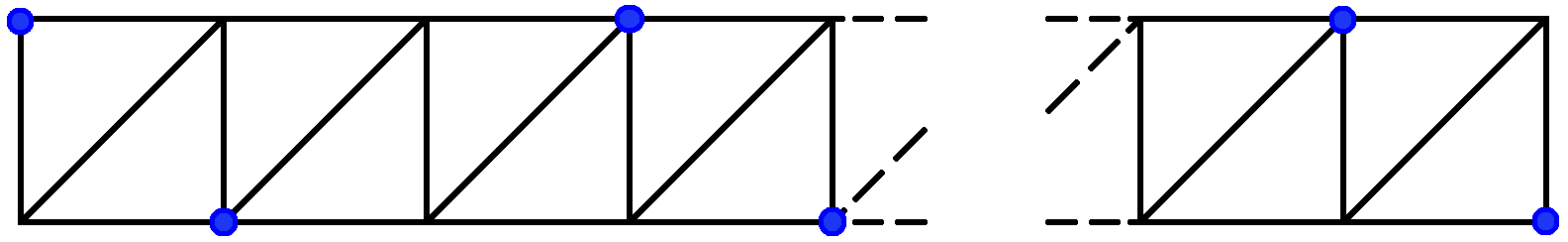}}
\put(10,2){\includegraphics[scale=.16]{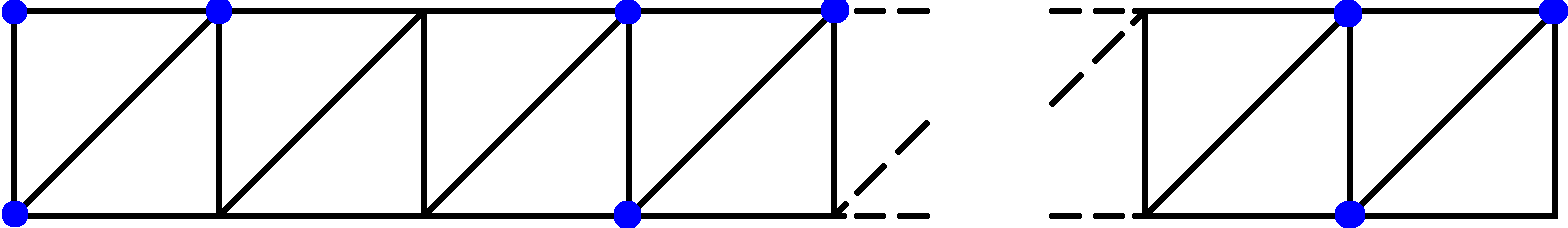}}
\end{picture}

We infer that the embedding of $\Fm$ into $C_g$ induces
a surjection on $H_1$ with $\ZZ_2$ coefficients and kernel
generated by the two elements above.

The quadratic form on the Milnor lattice is given by $1$ on the
generators. Its value on the elements of the radical can be
read off the diagrams, since this value is the $\mod 2$ Euler
number of the corresponding full subgraph.
In the first diagram above the subgraph
consists of an even number of isolated
vertices, in the second it consists of $d$ cycles.
Thus the quadratic form takes value $0$ on the radical.
We conclude that the monodromy group acts
on first homology of the trigonal fibre preserving the induced
quadratic form.

\paragraph{Remark}
The last claim can be deduced from \cite[thm.\ 11.1]{bh}.
Consider the basis of the first homology consisting of the
$c_1,\dots,c_{2g}$ (except for $c_{2g+1},c_{2g+2}$). 
Using the two elements
of the radical above, we can express $c_{2g+1}$,
$c_{2g+2}$ in this basis as
\begin{eqnarray*}
&&
c_1+c_2+c_3\: \:+ c_7+c_8+c_9 \: \:+\cdots 
+ c_{2g-1}+ c_{2g},\\
&&
c_1+ c_4 \: \:+ c_7 + c_{10}  \: \:+ \cdots + c_{2g-1}.
\end{eqnarray*}
The quadratic form defined on the basis elements to be $1$
has value $1$ also on these two elements, hence $(c)$ 
of  \cite[thm.\ 11.1]{bh} is
not given and the transvections on the monodromy
does not generate the full group of symplectic transformations.
Using \cite{j1,j2} it is then immediate that monodromy
acts by the full group of transformations respecting the 
quadratic form.

\section*{Appendix A: The construction of the monodromy map}

We include here a preliminary discussion of the monodromy map as this notion needs to be defined carefully to get our statements right.
\\

The families of trigonal curves are locally trivial in the complex topology.
Hence in our topological analysis we want to associate some
topological datum to it.
\\

Suppose $p:E\to B$ is a $G$-bundle with respect to an action of the group
$G$ on the fibre $F$. So $B$ is covered by open trivialization
patches $U$ with chart diagrams

\[\xymatrix{  F \times U \ar[r]^{\phi_U}_{\sim} \ar[d]_{pr_2} & E_{|U} \ar[d] \\
U \ar@{=}[r] & U } \]


such that a change of trivialization is given by a diagram

\[\xymatrix{ F \times U \ar[r]^{\phi^{-1}_V\phi_U} \ar[d] & F\times V \ar[d] \\
U & V } \]


where the map on top is given as

\[
(e,u) \mapsto (g_{U\cap V}(u)\cdot e, u)
\]
for $u\in U\cap V$ and some continuous $g:U\cap V \to G$. 
\\

Let $I$ be the unit interval. Since $I$ is contractible, a map $\gamma: I\to B$ can be lifted
to a morphism $\tilde\gamma: F\times I \to E$ of $G$-bundles. Let $V\subset I$ and $U\subset B$ two open sets such that $\gamma(V)\subset U$. Then
 $\tilde\gamma$ is trivialized locally on the base as follows

\[\xymatrix{ F \times V \ar[d] \ar[r]^{\tilde{\gamma}_{|F\times I}} & F \times U \ar[d] \\
V \ar[r]^{\gamma_{|V}} & U }\]

where the map on top is given by 

\[
(e,v) \mapsto ( g_{\phi_V,\phi_U}(u)\cdot e, \gamma(v))
\]

for some continuous $g:V \to G$.
\\

\begin{lemma}\label{nondep}
Given a closed path $\gamma$, i.e.\ $\gamma(0)=\gamma(1)$
the element
\[
g_{\phi_V,\phi_U}(0) g_{\phi_W,\phi_U}(1)\inv
\]
does not depend on the choice of a trivialization $\phi_U$ of 
$E$ at $\gamma(0)$.
\end{lemma}

\proof
A change of chart only introduces an element of $G$ and its
inverse between the two factors.
\qed

The natural consequence of Lemma \ref{nondep} that we care about is the following Proposition.

\begin{prop}
The following map given on representatives is well-defined
\[
\begin{matrix}
\pi_1(X, x_0 ) & \tto & \pi_0 ( G, id) \\
[\gamma] & \mapsto & 
[g_{\phi_V,\phi_U}(0) g_{\phi_W,\phi_U}(1)\inv]
\end{matrix}
\]
and is called the $G$-monodromy of the $G$-bundle $p:E\to X$.
\end{prop}

\bibliography{bibtrigo}

Michele Bolognesi, Institut de Recherche Math\'ematique de Rennes, 
Universit\'e de Rennes 1 \\ %
263 Avenue du G\'en\'eral Leclerc, 35042 Rennes Cedex, FRANCE.\\
 \hfill \texttt{\it E-mail: \rm michele.bolognesi@univ-rennes1.fr}

\smallskip

Michael L\"onne, Institut f\"ur Algebraische Geometrie, 
Gottfried Wilhelm Leibniz Universit\"at Hannover \\ 
Welfengarten 1, 30167 Hannover, GERMANY.\\
\hfill \texttt{\it E-mail: \rm loenne@math.uni-hannover.de}

\end{document}